\documentclass[12pt, oneside, leqno]{amsart}
\usepackage[T1]{fontenc}
\usepackage[latin2]{inputenc}
\usepackage{amsfonts}
\usepackage{amsmath,amsthm}
\usepackage{amssymb,latexsym}
\usepackage{enumerate}
\usepackage{graphicx}
\usepackage{xcolor}
\usepackage{tikz}

\newtheorem{thm}{Theorem}[section]
\newtheorem{cor}[thm]{Corollary}
\newtheorem{lem}[thm]{Lemma}

\newtheorem{prop}[thm]{Proposition}
\theoremstyle{definition}
\newtheorem{defin}[thm]{Definition}

\newtheorem{rem}[thm]{Remark}
\newtheorem{exa}[thm]{Example}

\numberwithin{equation}{section}

\newcommand{\ud}{\mathrm{d}}

\newcommand{\gP}{\mathbb{P}}
\newcommand{\gR}{\mathbb{R}}

\newcommand{\gN}{\mathbb{N}}
\newcommand{\gZ}{\mathbb{Z}}

\newcommand{\gE}{\mathbb{E}}
\newcommand{\bal}{\begin{align*}}
\newcommand{\eal}{\end{align*}}
\newcommand{\ball}{\begin{align}}
\newcommand{\eall}{\end{align}}
\definecolor{myGrey1}{gray}{0.2}
\definecolor{myGrey2}{gray}{0.4}
\definecolor{myGrey3}{gray}{0.6}
\definecolor{myGrey4}{gray}{0.8}
\begin{document}
\title{$\alpha$-stable random walk has massive thorns}
\author[A. BENDIKOV]{ALEXANDER BENDIKOV}
\thanks{Research of A. Bendikov was supported by National Science Centre, Poland, Grant DEC-2012/05/B/ST1/00613 and by SFB 701 of
German Research Council}
\author[W. CYGAN]{WOJCIECH CYGAN}
\thanks{Research of W. Cygan was supported by: National Science Centre, Poland, Grant DEC-2012/05/B/ST1/00613 and DEC-2013/11/N/ST1/03605; German Academic Exchange
Service (DAAD); SFB 701 of German Research Council}
\address{A. Bendikov \\
Institute of Mathematics\\
Wrocław University\\
50-384 Wrocław, Pl. Grunwaldzki 2/4, Poland}
\email{bendikov@math.uni.wroc.pl}
\address{W. Cygan \\
Institute of Mathematics\\
Wrocław University\\
50-384 Wrocław, Pl. Grunwaldzki 2/4, Poland}
\email{cygan@math.uni.wroc.pl}
\date{}

\begin{abstract}
We introduce and study a class of random walks defined on the integer
lattice $\mathbb{Z} ^d$ -- a discrete space and time counterpart of the
symmetric $\alpha$-stable process in $\mathbb{R} ^d$. When $0< \alpha <2$
any coordinate axis in $\mathbb{Z} ^d$, $d\geq 3$, is a non-massive set
whereas any cone is massive. We provide a necessary and sufficient condition
for the thorn to be a massive set.
\end{abstract}

\maketitle

\baselineskip=17pt


\renewcommand{\thefootnote}{}

\footnote{%
2010 \emph{\ Mathematics Subject Classification}: 31A15, 60J45, 05C81.}


\renewcommand{\thefootnote}{\arabic{footnote}} \setcounter{footnote}{0}



\section{Introduction}


\subsection*{Motivating questions}

This paper is motivated by the following two closely related questions.

1. Assuming that the probability $\phi $ on the group $\mathbb{Z} ^{d}$ is
symmetric and its support generates the whole $\mathbb{Z} ^{d}$, what is the
possible decay of the Green function 
\begin{equation*}
G(x)=\sum\limits_{n\geq 0}\phi ^{(n)}(x)
\end{equation*}%
as $x$ tend to infinity?

2. If $\phi $ is as above, which sets are massive/recurrent with respect to the random
walk driven by $\phi $?

Recall that the answer to the first question is known when $\phi $ is
symmetric, has finite second moment and $d\geq 3.$ Indeed, it is proved in 
\textsc{Spitzer} \cite{Spitzer} (see also \textsc{Saloff-Coste and Hebisch} 
\cite{Hebisch} for the treatment of general finitely generated groups) that $%
G(x)\sim c(\phi )\left\Vert x\right\Vert ^{2-d}$ at infinity. When the
second moment of $\phi $ is infinite but $\phi $ belongs to the domain of
attraction of the $\alpha$-stable law with $d/2<\alpha <\min \{d, 2\}$, $%
G(x)\sim c(\phi)\Vert x\Vert ^{\alpha -d}l(\Vert x\Vert)$ at infinity, where 
$l$ is an appropriately chosen slowly varying function, see \textsc{%
Williamson} \cite{Williamson}. However, there are many symmetric
probabilities $\phi $ for which the behaviour of the Green function $G$ at
infinity is not known.

In the present paper we use discrete subordination, a natural technique
developed in \textsc{Bendikov and Saloff-Coste} \cite{Lscbendikov} that
produces interesting examples of probabilities $\phi $ for which one can
estimate the behaviour of the Green function $G$ at infinity. This in turn
allows us to describe massiveness of some interesting classes of infinite
sets. For instance, we give necessary and sufficient conditions for the
thorn to be a massive set, see Section 4. Massiveness of thorns for the
simple random walk in $\mathbb{Z} ^d$, $d\geq 4$, was studied in the
celebrated paper of \textsc{Itô and McKean} \cite{McKean}.

The main idea behind this technique is the well-known idea of subordination
in the context of continuous time Markov semigroups but the applications we
have in mind require some adjustments and variations. The results we obtain
shed some light on the questions formulated above. The present paper is
concerned with examples when $\phi $ has neither finite support nor finite
second moment.

\subsection*{Subordinated random walks}

In the case of continuous time Markov processes, subordination is a
well-known and useful procedure of obtaining new process from an original
process. The new process may differ very much from the original process, but
the properties of this new process can be understood in terms of the
original process. The best known application of this concept is obtaining
the symmetric stable process from the Brownian motion. See e.g. \textsc{%
Bendikov} \cite{Bendikov}.

From a probabilistic point of view, a new process $(Y_{t})_{t>0}$ is
obtained from the original process $(X_{t})_{t>0}$ by setting $%
Y_{t}=X_{\varsigma _{t}}$, where the "subordinator" $(\varsigma _{s})_{s>0}$
is a nondecreasing Lévy process taking values in $(0,\infty )$ and
independent of $(X_{t})_{t>0}.$ See e.g. \textsc{Feller} \cite[Section X.7]%
{Feller}.

From an analytical point of view, the transition function $h_{\varsigma
}(t,x,B)$ of the new process is obtained as a time average of the transition
function $h(t,x,B)$ of the original process, that is,%
\begin{equation*}
h_{\varsigma }(t,x,B)=\int\limits_{0}^{\infty }h(s,x,B)\mathrm{d} \mu
_{t}(s).
\end{equation*}%
In this formula $\mu _{t}(s)$ is the distribution of the random
variable $\varsigma _{t}$. Subordination was first introduced by Bochner in
the context of semigroup theory. See \cite[footnote, p. 347]{Feller}.

Ignoring technical details, the minus infinitesimal generator $\mathcal{B}$
of the process $(Y_{t})_{t>0}$ is a function of the minus infinitesimal
generator $\mathcal{A}$ of the process $(X_{t})_{t>0}$, that is, $\mathcal{%
B=\psi (A)}.$ See \textsc{Jacob} \cite[Chapters 3 \& 4]{Jacob} for a
detailed discussion.

A discrete time version of subordination in which the functional calculus
equation $\mathcal{B=\psi (A)}$ serves as the defining starting point has
been considered by Bendikov and Saloff-Coste in \cite{Lscbendikov}. Given a
probability $\phi $ on $\mathbb{Z} ^{d}$ consider the random walk $%
X=\{X(n)\}_{n\geq 0}$ driven by $\phi $. In its simplest form, discrete
subordination is the consideration of a probability $\Phi $ defined as a
convex linear combination of the convolution powers $\phi ^{(n)}$. That is, 
\begin{gather*}
\Phi =\sum _{n\geq 1}c_n\, \phi ^{(n)},
\end{gather*}
where\ $c_n\geq 0$\ and $\sum _{n\geq 1}c_n =1$. We easily find that 
\begin{gather*}
\Phi ^{(n)}=\sum _{k\geq n}\Big( \sum _{k_1+\ldots +k_n=k}\, \, \prod
_{i=1}^{n}c_{k_i} \Big)\phi ^{(k)}.
\end{gather*}
The probabilistic interpretation is as follows: let $(R_i)$ be a sequence of
i.i.d. integer valued random variables, which are independent of $X$ and
such that $\mathbb{P} (R_i=k)=c_k$. Set $\tau _n=R_1+\ldots +R_n$, then 
\begin{gather*}
\mathbb{P} (\tau _n=k)=\sum _{k_1+\ldots +k_n=k}\, \, \prod _{i=1}^{n}c_{k_i}
\end{gather*}
and $\Phi ^{(n)}$ is the law of $Y(n)=X({\tau _n})$.

The other way to introduce the notion of discrete subordination is to use
Markov generators. Let $P$ be the operator of convolution by $\phi$. The
operator $L=I-P$ may be considered as minus the Markov generator of the
associated random walk. For a proper function $\psi$ we want to define a
"subordinated" random walk with Markov generator $-\psi (L)$. The
appropriate class of functions is the class of Bernstein functions, see the
book \textsc{Schilling, Song and Vondraček} \cite{BernsFunc}.

Recall that a function\ $\psi\in C^{\infty} (\gR ^+)$\ is called
a \textit{Bernstein function} if it is non-negative and $(-1)^{n-1}\psi
^{(n)}(x)\geq 0$, for all $x>0$ and all $n\in \gN $. The set of all Bernstein functions we denote by 
$\mathcal{BF}$. Each function $\psi \in \mathcal{BF}$ has
the following representation 
\begin{align}
\psi (\theta )=a+b\theta+\int _{(0,\infty )}\, \big( 1-e^{-\theta s} \big) %
\, \mathrm{d} \nu (s),  \label{levkch}
\end{align}
for some constants $a,b\geq 0$ and some measure $\nu$ (the L\'{e}vy measure) such
that 
\begin{gather*}
\int _{(0, \infty )}\, \min \{ 1, s\}\, \mathrm{d} \nu (s)< \infty .
\end{gather*}
\begin{prop}\label{prop111}
\cite[Proposition 2.3]{Lscbendikov}\label{prop1} Assume that\ $\psi$\ is a
Bernstein function with its representation (\ref{levkch}),\ such that\ $\psi
(0)=0,\, \psi (1)=1$\ and set 
\begin{equation}  \label{coefCpsi}
\begin{split}
c(\psi ,1)&=b+\int _{(0,\infty)} \, t e^{-t}\, \mathrm{d} \nu (t), \\
c(\psi ,n)&=\frac{1}{n! }\int _{(0,\infty)} \, t^n e^{-t}\, \mathrm{d} \nu
(t),\ n>1 .
\end{split}%
\end{equation}
Let\ $\phi$ be a probability on $\mathbb{Z} ^d$. Let\ $P$\ be the operator
of convolution by \ $\phi$\ and set 
\begin{gather}
P_\psi =I-\psi (I-P)\, .  \label{opPPsi}
\end{gather}
Then $P_\psi $ is the convolution by a probability $\Phi $\ defined as 
\begin{gather}
\Phi =\sum _{n\geq 1}c(\psi ,n)\, \phi ^{(n)}.  \label{fipsidens}
\end{gather}
\end{prop}
\begin{exa}
The power function
$\psi _\alpha (s)=s^{\alpha /2} $, $\alpha \in (0,2)$ belongs to the class $\mathcal{BF}$.
Its L\'{e}vy density $\nu_{\alpha}(t)$ is given by
\begin{align*}
\nu_{\alpha}(t)=\frac{\alpha/2}{\Gamma(1-\alpha/2)} t^{-1-\alpha/2}.
\end{align*}
The probabilities $c(\psi_{\alpha},n)$ are given by
\begin{align*}
c(\psi _\alpha, n)=\frac{\alpha /2}{\Gamma (1-\alpha /2)}\frac{\Gamma (n-\alpha /2)}{\Gamma (n+1)}\sim \frac{\alpha /2}{\Gamma (1-\alpha /2)} n^{-1-\alpha /2} .	
\end{align*}
Choosing $\psi =\psi _\alpha$ in Proposition \ref{prop111}, we see that the Markov generators of the initial and new random walks are related by the equation
\begin{align*}
I-P_{\psi _\alpha }=(I-P)^{\alpha /2}.
\end{align*}
\end{exa}

\begin{defin}
\label{defSubord} Let $X=\{X(n)\}_{n\geq 0}$ be the random walk driven by $%
\phi$. The random walk with the transition operator $P_\psi $ defined at (%
\ref{opPPsi}) will be called the\textit{\ $\psi$-subordinated random walk}
and will be denoted by $X_{\psi}=\{X_\psi (n)\}_{n\geq 0}$. When $\psi =
\psi _\alpha$ and $X=S$ is the simple
random walk, we call $X_\psi$ the \textit{$\alpha$-stable
random walk} and denote it by $S_\alpha$.
\end{defin}

It is straightforward to show that the increments of $S_\alpha$ belong to
the domain of attraction of the $\alpha$-stable law. This fact justifies the
name "$\alpha$-stable random walk" given in the Definition \ref{defSubord}.

\subsection*{Notation}

For any two non-negative functions $f$ and $g$, $f(r)\sim g(r)$ at $a$ means
that $\lim _{r\rightarrow a}f(r)/g(r)=1$, $f(x)=O(g(x))$ if $f(x)\leq Cg(x)$%
, for some constant $C>0$, and $f(x)\asymp g(x)$ if $f(x)=O(g(x))$ and $%
g(x)=O(f(x))$. 

\section{Green function asymptotic}

Let $S$ be the simple random walk and $\psi\in\mathcal{BF}$. Assuming that the subordinated random walk $S_{\psi}$
is transient we study asymptotic behaviour of its Green function
$G_{\psi}$.

In the course of study we will use the following technical assumption:
the function $\psi $ satisfies $\psi(0)=0$, $\psi(1)=1$ and
\begin{equation}
\psi (\lambda )= \lambda ^{\alpha /2}/l(1/\lambda ),\label{rvcbf1}
\end{equation}%
where $0<\alpha <2$ and $l(\lambda )$ varies slowly at infinity.

Recall that a function $f$ defined in a neighbourhood of $0$ is said to \textit{vary
regularly of index $\beta $} at $0$ if for all$\ \lambda >1,$ 
\begin{equation*}
\lim_{x\rightarrow 0}\frac{f(\lambda x)}{f(x)}=\lambda ^{\beta }.
\end{equation*}
When $\beta =0$, one says that $f$ \textit{varies slowly} at $0$. Any
regularly varying function of index $\beta $ is of the form $f(x)=x^{\beta
}l(x)$, where $l$ is a slowly varying function. For example, each of the
following functions vary regularly at $0$ of index $\beta $: $x^{\beta
}\left( \log 1/x\right) ^{\delta }\!\!\!,\ x^{\beta }\exp \{\left( \log
1/x\right) ^{\delta }\},\ 0<\delta <1,$ etc.

A function $F$ defined in a neighbourhood of $\infty$ is said to \textit{%
vary regularly of index $\beta $} at $\infty $ if $f(x)=F(1/x)$ varies 
regularly of index $-\beta $ at $0.$\\
\par 
Let $c(\psi ,k)$, $k\in \gN$, be the probabilities defined at (\ref{coefCpsi}). For $k\leq 0$ we set
$c(\psi,k)=0$ and consider $\tau=(\tau_n)_{n\geq 0}$ -- random walk
on $\gZ$ whose increments $\tau _{n+1}-\tau _n$ have distribution $c =\{c(\psi ,k)\}_{k\in \gZ}$. The random walk $\tau$
has non-negative increments, in particular it is transient. Let
\begin{align*}
C(B)=\sum_{k\geq 0} c^{(k)} (B),\quad B\subset \gZ
\end{align*}
be its potential measure; here $c^{(k)}$ is the Dirac measure
concentrated at $0$ when $k=0$ and the $k$-fold convolution of the
probability $c$ when $k\geq 1$. Setting $C(n)=C(\{n\})$ we obtain $C(n)=0$ for $n<0$,
$C(0)=1$ and
\begin{align*}
C(n)=\sum_{k=1}^{n} c(\psi,k)C(n-k),\quad n \geq 1.
\end{align*}

Recall that a function $\psi\in \mathcal{BF}$ is called a special Bernstein function, in short
$\psi\in \mathcal{SBF}$, if the function $\lambda \big/ \psi (\lambda) $ is also a Bernstein function. 
Evidently $\psi _\alpha \in \mathcal{SBF}$ whereas $\psi (\lambda )=1-e^{-\lambda }$ does not belong to $\mathcal{SBF}$. In particular, $\mathcal{SBF}\subset \mathcal{BF}$ is a proper inclusion.
\begin{lem}\label{Lemma111}
Let $\psi \in \mathcal{BF}$ satisfy (\ref{rvcbf1}). The strong renewal property  
\begin{equation}\label{c_kRel}
C(n)\sim \frac{1}{\Gamma (\alpha /2 )}n^{\alpha /2-1}l(n),\quad n\rightarrow \infty ,
\end{equation}
holds in the following two cases:
\begin{itemize}
\item[(i)] $\psi \in \mathcal{BF}$ and $1<\alpha <2$,
\item[(ii)] $\psi \in \mathcal{SBF}$ and $0<\alpha <2$.
\end{itemize}
\end{lem}

\begin{proof}
Define an auxiliary function $M(x)$, $x\in \gR$, as
\begin{equation*}
M(x)=\sum_{k\leq x}C(k).
\end{equation*}
Observe that $M$ is a right continuous step-function having jumps at integers.
More precisely $M(x)=0$ for $x<0$, $M(x)=C(0)$ for $0\leq x < 1$, $M(x)= C(0)+C(1)$ for
$1\leq x < 2$ etc. We compute the Laplace-Stieltjes transform $\mathcal{L}(M)$
of the function $M$, 
\begin{align}
\mathcal{L}(M)(\lambda )& =\int_{\gR }e^{-\lambda x}\mathrm{d}M(x)=\sum_{k=0}^{\infty }e^{-\lambda k}C(k)\label{LMc(k)}\\
& =\sum_{k=0}^{\infty }\sum _{n=0}^\infty e^{-\lambda k}c ^{(n)}(k)
=\sum_{n=0}^{\infty
}\sum_{k=0}^{\infty }e^{-\lambda k}\,\mathbb{P}(\tau _{n}=k) \nonumber\\
& =\sum_{n=0}^{\infty }\Big(\mathbb{E}\big(e^{-\lambda \tau _{1}}\big)\Big)%
^{n}=\frac{1}{1-\mathbb{E}\big(e^{-\lambda \tau _{1}}\big)}\nonumber.
\end{align}%
We claim that, 
\begin{equation}
\mathbb{E}\big(e^{-\lambda \tau _{1}}\big)=1-\psi \big(1-e^{-\lambda }\big)%
\,.\label{claim1}
\end{equation}%
Indeed, by Proposition \ref{coefCpsi}, 
\begin{equation*}
\mathbb{E}\big(e^{-\lambda \tau _{1}}\big)=\sum_{k=1}^{\infty }e^{-\lambda
k}\,c(\psi ,k).
\end{equation*}%
Using (\ref{levkch}) and the fact that $\psi (1)=1$ we obtain 
\begin{align*}
1-\psi (1-e^{-\lambda })& =1-b(1-e^{-\lambda })-\int_{(0,\infty )}\,\big(%
1-e^{-t(1-e^{-\lambda })}\big)\,\mathrm{d}\nu (t) \\
& =1-\Big(b+\int_{(0,\infty )}\,(1-e^{-t})\,\mathrm{d}\nu (t)\Big)%
+be^{-\lambda } \\
& \qquad +\int_{(0,\infty )}\,e^{-t}\sum_{n=1}^{\infty }\frac{%
t^{n}e^{-n\lambda }}{n!}\,\mathrm{d}\nu (t) \\
& =be^{-\lambda }+\sum_{n\geq 1}\frac{1}{n!}\Big(\int_{(0, \infty )}\,e^{-t}t^{n}\,\mathrm{d}\nu (t)\Big)e^{-\lambda n}=\sum_{n\geq 1}c(\psi
,n)e^{-\lambda n},
\end{align*}%
as desired. It follows that  
\begin{equation*}
\mathcal{L}M(\lambda )=\frac{1}{\psi \big(1-e^{-\lambda }\big)}.
\end{equation*}%
Hence, by (\ref{rvcbf1}) we obtain 
\begin{equation*}
\mathcal{L}M(\lambda )\sim \lambda ^{-\alpha /2}l(1/\lambda ),\quad \mathrm{as}\ \lambda \to 0^+.
\end{equation*}%
By the Karamata's Tauberian Theorem \cite[Theorem 1.7.1]{Bingham}, 
\begin{equation}
M(x)\sim \frac{1}{\Gamma \big(1+\frac{\alpha }{2}\big)}x^{\frac{\alpha }{2}%
}\,l(x),\quad \mathrm{as}\ x\rightarrow \infty. \label{integrated potential}
\end{equation}%
By \cite[Theorem 8.7.3]{Bingham}, the equation (\ref{integrated potential}) is equivalent to
\begin{align}
\sum _{k=n}^\infty c(\psi , k)\sim \frac{n^{-\alpha /2}}{l(n)\Gamma\big(1-\frac{\alpha}{2}\big)},\quad \mathrm{as}\ n\to \infty .\label{sum_c(psi)}
\end{align}
Moreover, recall that
\begin{align}
C(0)=1,\qquad C(n)=\sum_{k=1}^{n} c(\psi,k)C(n-k),\quad n > 1.\label{renewal_recur}
\end{align}
The celebrated Garsia-Lamperti theorem \cite[Theorem 1.1]{Garsia} says that (\ref{renewal_recur}) and (\ref{sum_c(psi)}) imply that, when $1<\alpha <2$,
\begin{align*}
C(n)\sim \Gamma\Big(1-\frac{\alpha}{2}\Big)\frac{\sin \big(\pi \alpha /2\big)}{\pi }n^{\alpha /2 -1}l(n),\quad \mathrm{as}\ n\to \infty .
\end{align*}
Using the Euler's reflection formula
\begin{align*}
\Gamma(z)\Gamma(1-z)=\frac{\pi}{\sin (\pi z)} 
\end{align*}
 we obtain (\ref{c_kRel}).

Let us pass to the proof of (ii).
Since $\psi \in \mathcal{SBF}$, we have
\begin{align}
\frac{1}{\psi (\lambda)}=b + \int _0^\infty e^{-\lambda t}u(t) \ud t
\end{align}
for some $b\geq 0$ and some non-increasing function $u\colon (0,\infty )\mapsto (0,\infty )$ satisfying $\int _0^1u(t)\ud t<\infty$, see \cite[Theorem 11.3]{BernsFunc}. Set $\Phi (\lambda )=1\big/ \psi (\lambda )$ and observe that by (\ref{rvcbf1}), 
\begin{align*}
\Phi (\lambda)\sim \lambda ^{-\alpha /2}l(1/\lambda),\quad \lambda \to 0.
\end{align*}
Applying both the Karamata Tauberian Theorem \cite[Theorem 1.7.1]{Bingham}
and the Monotone Density Theorem we obtain
\begin{align}
u(t)\sim \frac{1}{\Gamma (\alpha /2)}t^{\alpha /2-1}l(t),\quad t\to \infty .\label{asympu}
\end{align}
On the other hand
\begin{align*}
\mathcal{L}(M)(\lambda )=\Phi (1-e^{-\lambda})=\sum _{k\geq 0}\frac{(-1)^k\Phi ^{(k)}(1)}{k!}e^{-\lambda k},
\end{align*}
whence by the uniqueness of the Laplace transform we obtain
\begin{align*}
C(k)=\frac{1}{k!}\int _0^\infty t^ke^{-t}u(t)\ud t, \quad k\in \gN.
\end{align*}
We claim that
\begin{align*}
C(k)=\frac{1}{k!}\int _{k/2}^{2k} t^ke^{-t}u(t)\ud t +O((2/e)^{k}).
\end{align*}
To prove the claim observe that the function $t\mapsto t^ke^{-t}$ is unimodal with $\max$ at the point $t=k$.
Hence for $a,b$ and $k$ large enough we will have
\begin{align*}
\int _1^a t^ke^{-t}u(t)\ud t \leq a^ke^{-a}\int _1^a u(t) \ud t,\quad a<k 
\end{align*}
and
\begin{align*}
\int _b^\infty t^ke^{-t}u(t)\ud t \leq b^{k+1}e^{-b}\int_b^\infty \frac{u(t)}{t}\ud t
,\quad b>k .
\end{align*}
In particular, choosing $a=k/2$, $b=2k$ and applying (\ref{asympu}) we obtain
\begin{align*}
\frac{1}{k!}\Big( \int _0^{k/2} t^ke^{-t}u(t)\ud t + \int _{2k}^\infty t^ke^{-t}u(t)\ud t \Big)= O\big((2/e)^k\big),
\end{align*}
which evidently proves the claim.

Once again applying (\ref{asympu}) we get
\begin{align*}
\frac{1}{k!}\int _{k/2}^{2k} t^ke^{-t}u(t)\ud t&\sim \frac{l(k)}{k!\Gamma (\alpha /2)}\int _{k/2}^{2k} t^{k+\alpha /2-1}e^{-t}\ud t.
\end{align*}
It is straightforward to show that 
\begin{align*}
\frac{1}{k!}\int _{k/2}^{2k} t^{k+\alpha /2-1}e^{-t}\ud t
= \frac{1}{k!}\int _0^\infty  t^{k+\alpha /2-1}e^{-t}\ud t + O\big((2/e)^{k}\big).
\end{align*}
At last, all the above show that
\begin{align*}
C(k)\sim \frac{l(k)\Gamma (k+\alpha /2)}{\Gamma (\alpha /2)\Gamma (k+1)}
\sim \frac{1}{\Gamma (\alpha /2) }k^{\alpha /2 -1}l(k).
\end{align*} 
The proof of (ii) is finished.
\end{proof}
\begin{rem} 
Remember that in the continuous time setting to each function $\psi \in \mathcal{BF}$ is associated a unique
convolution semigroup $(\eta _t)_{t>0}$ of measures supported on $[0,\infty )$ such that 
\begin{align*}
\mathcal{L}\eta _t(\lambda) = e^{-t\psi (\lambda )}.
\end{align*}
A function $\psi \in \mathcal{SBF}$ is characterized by the fact that the potential measure $U=\int_0^\infty \eta _t\ud t$ restricted to $(0,\infty )$ is absolutely continuous with respect to the Lebesgue measure and its density $u(t)$ is a decreasing function. \textit{Whether this is true in the discrete time setting, i.e. the sequence $C(k)$ is decreasing, is an open question at the present writing}.
\end{rem}

We present here some partial answer to this question. 
Recall that a function $\psi\in \mathcal{BF}$ is called a complete Bernstein function,
$\psi\in \mathcal{CBF}$ in short, if its L\'{e}vy measure $\nu$ is absolutely continuous with
respect to the Lebesgue measure and its density $\nu(s)$ is completely
monotone, i.e.
\begin{align*}
\nu(s)=\int_{[0, \infty )} e^{-st} \mu(dt),
\end{align*}
Observe that in fact $\mu $ is supported on $(0,\infty )$ and satisfies 
\begin{align*}
\int_{(0, \infty )} \min(t^{-1},t^{-2}) \mu(dt)<\infty .
\end{align*} 
 $\mathcal{CBF}\subset \mathcal{SBF}$ is a proper inclusion. For all of this we refer to \cite{BernsFunc}.
\begin{thm}
For $\psi \in \mathcal{CBF}$ the renewal sequence $\{C(k)\}_{k\in \gN}$ defined as
\begin{align*}
C(0)=1\quad \mathrm{and}\quad C(k)=\sum _{n=0}^k c(\psi ,n)C(k-n),\quad k\geq 1,
\end{align*}
is decreasing.
\end{thm}
\begin{proof}
We give a proof of the statement in four steps.\\
\textit{Claim 1}. There exist a measure $m$ on $(0,\infty)$ such that
\begin{align*}
c(\psi ,1) = b+\int _{(0,\infty )}e^{-2r} m(\ud r)
\end{align*}
and
\begin{align*}
c(\psi,n)=\int_{(0, \infty )} e^{-(n+1)r} m(\ud r),\quad n>1.
\end{align*}
We consider the case $n>1$. Since $\psi\in\mathcal{CBF}$, 
\begin{align*}
 c(\psi,n)&=\frac{1}{n!}\int_{(0, \infty )} t^n e^{-t} \nu(t)\ud t\\
   &= \frac{1}{n!}\int_{(0, \infty )}\ud  t\, t^n e^{-t}\int_{(0,\infty )}e^{-st}\mu(\ud s)\\
   &=\int_{(0, \infty )}\mu(\ud s)\frac{1}{n!}\int_{(0,\infty )} t^n e^{-t(1+s)}\ud t
   =\int_{(0, \infty )}\frac{\mu(\ud s)}{(1+s)^{n+1}}.
\end{align*}
Substitution $\log (1+s)=r$ gives
\begin{align*}
    c(\psi,n)=\int_{(0, \infty )} e^{-(n+1)r} m(\ud r),
\end{align*}
as desired.\\
\textit{Claim 2}. $ \{c(\psi ,n)\}_{n\in \gN} $ satisfies
\begin{align*}
c(\psi,n-1)c(\psi,n+1)>c(\psi,n)^2,\quad n>1.
\end{align*}
It is enough to consider the case $n>2$. We apply Claim 1,
\begin{align*}
c(\psi,n-1)c(\psi,n+1)&=
\int_{(0, \infty )} m(\ud s)\int_{(0, \infty )} m(\ud t) e^{-\{ns+(n+2)t\}}\\
&=\int_{(0, \infty )} m(\ud s)\int_{(0, \infty )} m(\ud t)  e^{-(n+1)(s+t)}e^{s-t}\\
&=\int_{(0, \infty )} m(\ud s)\int_{(0, \infty )} m(\ud t)  e^{-(n+1)(s+t)}\cosh (s-t)\\
&> \int_{(0, \infty )} m(\ud s)\int_{(0, \infty )} m(\ud t)  e^{-(n+1)(s+t)}
=c(\psi,n)^2.
\end{align*}
The strong inequality follows from the fact that, by (\ref{rvcbf1}), $m$ is not a Dirac measure.\\
\textit{Claim 3}. $\{C(n)\}_{n\in \gN}$ satisfies
\begin{align*}
C(n-1)C(n+1)>C(n)^2,\quad n>1.
\end{align*}
Indeed, we have 
\begin{align*}
C(n)=\sum_{k=1}^{n} c(\psi,k)C(n-k),\quad n\geq 1
\end{align*}
and
\begin{align*}
c(\psi,n-1)c(\psi,n+1)>c(\psi,n)^2,\quad n>1.
\end{align*}
The remarkable de Bruijn-Erd\"{o}s theorem \cite[Theorem 1]{Erdos} yields the desired result.

Finally we prove that $\{C(k)\}_{k\in \gN}$ is a decreasing sequence. By
Claim 3, the sequence $C(k+1)/C(k)$ increases. Assume that $C(k_0+1)/C(k_0)\geq
1$, for some $k_0\in \gN$. Then there are some $a>1$ and $N\in \gN$ such that $C(k+1)/C(k)>a$ for all
$k\geq N$. It follows that $C(n)\geq a^{n-N}C(N)$, for all $n\geq N$.
Contradiction, because for all $n\geq 0$,
\begin{align*}
C(n)=\gP(\exists k: \tau _k=n)\leq 1.
\end{align*}
Thus $C(k)$ decreases. 
\end{proof}
Let $p(n,x)$ be a transition function of the simple random walk $S$. By $p_\psi (n,x)$ we denote a transition function of the subordinated random walk $S_\psi$ and by $G_{\psi }$ its Green function,
\begin{align*}
p_\psi (n,x) = \sum _{k = 1}^{\infty }p(k,x)\gP (\tau _n = k)
\end{align*}
and
\begin{eqnarray*}
G_{\psi }(x) =\sum_{n=1}^{\infty }p_{\psi }(n,x)  =\sum_{k=1}^{\infty }p(k,x)\,C(k).
\end{eqnarray*}

\begin{thm}
\label{Asymptotic} Assume that $\psi \in \mathcal{BF}$ satisfies (\ref{rvcbf1}) with $0<\alpha <d$ and that (\ref{c_kRel}) holds. Then 
\begin{equation*}
G_{\psi }(x)\sim \frac{C_{d,\alpha }}{\Vert x\Vert ^{d}\psi (1/\Vert x\Vert
^{2})},\quad x\to \infty ,
\end{equation*}%
where 
\begin{equation*}
C_{d,\alpha }=\Big(\frac{d}{2}\Big)^{\alpha /2}\frac{\,\,\pi ^{-d/2}}{\Gamma %
\big(\frac{\alpha }{2}\big)}\,\Gamma \Big(\frac{d-\alpha }{2}\Big).
\end{equation*}
\end{thm}

\begin{proof}
Remember that $p(k,x)$ is the $k$-step transition probability of the simple random walk started at 0.  
Since $p(k,x)=0$\ for\ $k< \frac{\Vert x\Vert}{\sqrt {d}}$, we have
\begin{align*}
G_{\psi }(x)&=\sum _{k\geq \frac{\Vert x\Vert }{\sqrt{d}}} C(k)\, p(k,x) \\
&=\underbrace{\sum _{k> \frac{\Vert x\Vert ^2}{A}}\! \! \! C(k)\, p(k,x)}_{=I_1}\ +\underbrace{\sum _{\frac{\Vert x\Vert }{\sqrt {d}}\leq k\leq \frac{\Vert x\Vert ^2}{A}}\! \! \! \! \! \! \! C(k)\, p(k,x)}_{=I_2} ,
\end{align*}
where\ $A>1$\ is a constant which will be specified later.

Our further analysis is based on the results of G.F.~Lawler \cite[%
Section 1.2]{Lawler}. We write $n\leftrightarrow x$ when $n+x_{1}+...+x_{d}$ is \emph{even}. 
Set
\begin{align*}
q (n,x)=2\Big( \frac{d}{2\pi n}\Big) ^{\frac{d}{2}}e^{-\frac{d\Vert x\Vert^{2}}{2n}}
\end{align*}
and define the error function 
\begin{equation*}
E(n,x)=\left\{ 
\begin{array}{ccc}
p(n,x)-q(n,x) & \text{if} & n\leftrightarrow x, \\ 
0 & \text{if} & n\nleftrightarrow x.%
\end{array}%
\right. 
\end{equation*}%
By \cite[Theorem 1.2.1]{Lawler}, 
\begin{equation}
\left\vert E(k,x)\right\vert \leq c_{1}\Vert x\Vert ^{-2}k^{-d/2},\label{lclt}
\end{equation}%
for some $c_{1}>0$ and all $k\geq 1$.

To study $I_1$ we may assume that $x\leftrightarrow 0$, then $p(2k+1,x)=0,$ for all $k\geq 1.$
Writing $I_1$ in the form,
\begin{gather*}
I_1=\underbrace{\sum _{2k> \frac{\Vert x\Vert ^2}{A}}\! \! \! C(2k)\, q (2k,x)}_{I_{11}}\  + \underbrace{\sum _{2k> \frac{\Vert x\Vert ^2}{A}}\! \! \! C(2k)\, E(2k,x)}_{I_{12}} 
\end{gather*}
and using\ (\ref{c_kRel}) and\ (\ref{lclt}) we obtain 
\begin{align}
I_{12}&\leq c_2\sum _{k> \frac{\Vert x\Vert ^2}{A}}\! \! \! k^{\alpha /2-1}\, l(k\, ) \frac{k^{-d/2}}{\Vert x\Vert ^2} \nonumber \\ 
 &\sim  c_2
 \int _{\frac{\Vert x\Vert ^2}{A}}^{\infty}\ t^{\alpha /2-d/2-1}\, l(t)\, \ud t   \quad \textrm{as}\ x\rightarrow \infty ,\label{a.s.g7}
\end{align}
for some constant $c_2>0$.
By \cite[Proposition 1.5.10]{Bingham}, 
\begin{align*}
 \int _{\frac{\Vert x\Vert ^2}{A}}^{\infty}\ t^{\alpha /2-d/2-1}\, l(t)\, \ud t \sim  
 \frac{2}{d-\alpha}\, A^{\frac{d-\alpha }{2}}\, \Vert x\Vert ^{\alpha -d}\, l( \Vert x\Vert ^2)\quad \textrm{as}\ x\rightarrow \infty .
\end{align*}
It follows that
\begin{gather*}
\lim _{\Vert x \Vert \rightarrow \infty}\, \frac{\Vert x \Vert ^{d-\alpha}}{l(\Vert x \Vert ^2)}I_{12}=0 .
\end{gather*}
Similarly, when $\Vert x\Vert \rightarrow \infty$,
\begin{align*}
I_{11}&\sim \frac{2\Big ( \frac{d}{2\pi}\Big) ^{d/2}}{\Gamma \big( \frac{\alpha}{2}\big)}\sum _{2k> \frac{\Vert x\Vert ^2}{A}}\! \! \! (2k)^{\alpha /2-1}\, l(2k)\, (2k)^{-d/2}\, \exp \left\{ \frac{-d\Vert x\Vert ^2}{4k} \right\} \\
 &\sim  \frac{\Big ( \frac{d}{2\pi}\Big) ^{d/2}}{\Gamma \big( \frac{\alpha}{2}\big)}\int _{\frac{\Vert x\Vert ^2}{A}}^{\infty}\, t^{\alpha /2-d/2-1}\, \exp\left\{ \frac{-d\Vert x\Vert ^2}{2t} \right\}\, l(t)\, \ud t . 
\end{align*}
Applying \cite[Proposition 4.1.2]{Bingham} we obtain
\begin{gather*}
I_{11}\sim  \Big( \frac{d}{2}\Big)^{\alpha /2} \frac{\, \, \pi ^{-d/2}}{\Gamma \big( \frac{\alpha}{2}\big) }\, \Vert x\Vert^{\alpha -d}\, l(\Vert x\Vert ^2)\int _{0}^{Ad/2}\, s^{d/2-\alpha /2-1}\, e^{-s}\, \ud s .
\end{gather*}
It follows that
\begin{gather*}
\lim _{\Vert x \Vert \rightarrow \infty}\, \frac{\Vert x \Vert ^{d-\alpha}}{l(\Vert x \Vert ^2)}I_{11}= \Big( \frac{d}{2}\Big)^{\alpha /2} \frac{\, \, \pi ^{-d/2}}{\Gamma \big( \frac{\alpha}{2}\big) }\int _{0}^{dA/2}\, s^{d/2-\alpha /2-1}\, e^{-s}\, \ud s:=C_1(A).\label{a.s.g11}
\end{gather*}
\par To estimate\ $I_2$ we use the Gaussian upper bound from \cite[Theorem 2.1]{Hebisch},
\begin{align*}
I_2&\leq  c_3 \sum _{\frac{\Vert x\Vert}{\sqrt {d}}\leq k\leq \frac{\Vert x\Vert^2}{A}}\! \! \! \! \! \! \! k^{\alpha /2-1}\, l(k)\, k^{-d/2}\, \exp \left\{ \frac{-\Vert x\Vert^2}{c_4k}\right\}  \\
 &\sim  c_3\int _{\frac{\Vert x\Vert}{\sqrt {d}}}^{\frac{\Vert x\Vert^2}{A}} t^{\alpha /2-d/2-1}\, \exp \left\{ \frac{-\Vert x\Vert^2}{c_4t}\right\}\, l(t)\, \ud t \\
 &= \frac{c_3}{c_4^{\alpha /2-d/2}}\, \Vert x\Vert^{\alpha -d}\int _{\frac{A}{c_4}}^{\frac{\sqrt {d}\Vert x\Vert}{c_4}} s^{d/2-\alpha /2-1}\, e^{-s}\, l\Big(\frac{\Vert x\Vert ^2}{c_4 s} \Big)\, \ud s  \\
 &\leq   \frac{c_3}{c_4^{\alpha /2-d/2}}\, \Vert x\Vert^{\alpha -d}\int _{\frac{A}{c_4}}^{\infty} s^{d/2-\alpha /2-1}\, e^{-s}\, l\Big(\frac{\Vert x\Vert ^2}{c_4 s} \Big)\, \ud s  ,
\end{align*}
for some constants $c_3, c_4>0$.
Next we apply \cite[Theorem 1.5.6]{Bingham} and the Dominated Convergence Theorem 
\begin{gather*}
\limsup _{\Vert x\Vert\rightarrow \infty }\, \frac{\Vert x\Vert^{d-\alpha}}{l(\Vert x\Vert ^2)}\, I_2\leq  \frac{c_3}{c_4^{\alpha /2-d/2}\, \Gamma \big( \frac{\alpha}{2}\big)}\int _{\frac{A}{c_2}}^{\infty } s^{d/2-\alpha /2-1}\ e^{-s}\, \ud s := C_2(A).
\end{gather*}
All the above show that, for any fixed $A>1$,
\begin{align*}
\limsup _{\Vert x\Vert\rightarrow \infty }\, \frac{\Vert x\Vert^{d-\alpha}}{l(\Vert x\Vert ^2)}\, G_\psi (x)\leq  C_1(A)+C_2(A)  
\end{align*}
and
\begin{align*}
\liminf _{\Vert x\Vert\rightarrow \infty }\, \frac{\Vert x\Vert^{d-\alpha}}{l(\Vert x\Vert ^2)}\, G_\psi (x)&\geq  C_1(A) .
\end{align*}
At last, we have
\begin{align*}
\lim _{A\rightarrow \infty }\, C_2(A)=0
\end{align*}
and
\begin{gather*}
\lim _{A\rightarrow \infty }\, C_1(A)= \Big( \frac{d}{2}\Big)^{\alpha /2} \frac{\, \, \pi ^{-d/2}}{\Gamma \big( \frac{\alpha}{2}\big) }\, \Gamma \Big( \frac{d-\alpha}{2}\Big) . \label{c1a2}
\end{gather*}
The proof is finished.
\end{proof}
\begin{rem}
One useful observation is that if we assume that the function $\psi$ satisfies
\begin{equation*}
\psi (\theta )\asymp \theta ^{\alpha /2}/l(1/\theta )\quad \text{at}\ 0
\end{equation*}%
and belongs to the class $\mathcal{SBF}$,
then following the line of reasons of Lemma \ref{Lemma111} and Theorem \ref{Asymptotic}, we obtain that
\begin{align*}
C(k)\asymp k^{\alpha /2-1}l(k)\quad \text{at}\ \infty
\end{align*} 
and 
\begin{equation*}
G_{\psi }(x)\asymp \Vert x\Vert ^{\alpha -d}l(\Vert x\Vert ^{2})\quad \text{%
at}\ \infty .
\end{equation*}
Whether this is true when $\psi\in \mathcal{BF} \setminus \mathcal{SBF}$ is an open question at
the present writing. In the closely related paper \cite{Molchanov} some partial results in this direction are obtained.
\end{rem}


\section{Massive sets}

\subsection*{Basic definitions.}

Let $X=\{X(n)\}_{n\geq 0}$ be a transient random walk on $\mathbb{Z} ^d$.
Let $B$ be a proper subset of $\mathbb{Z} ^d$ and $p _B $ the hitting
probability of $B$. The set $B$ is called \textit{massive}/\textit{recurrent} if $p _B (x)=1$\
for all\ $x\in \mathbb{Z} ^d$\ and \textit{non-massive} otherwise.

Let $\pi _B (x)$ be the probability that the random walk $X$ starting from $%
x $ visits the set $B$ infinitely many times. The set $B$ is massive if and
only if $\pi _B \equiv 1$; for non-massive $B$, $\pi _B $ is identically $0$.

Let $G(x,y)$ be the Green function of $X$. In general, the function $p _B$
is excessive, whence it can be written in the form 
\begin{gather*}
p _B=G \varrho _B + \pi _B .
\end{gather*}
When $B$ is a non-massive set, i.e. $\pi _B\equiv 0$, $p_B$ is a potential.
It is called the \textit{equilibrium potential} of $B$, respectively $\varrho
_B $ - the \textit{equilibrium distribution}. When $B$ is non-massive, the 
\textit{capacity} of $B$ is defined as 
\begin{align*}
Cap (B)=\sum _{y\in B}\varrho _B (y).
\end{align*}
The quantity $Cap(B)$ can be also computed as 
\begin{gather*}
Cap (B)=\sup \, \{\, \sum _{y\in B}\varrho (y):\, \varrho \in \Xi _B\} ,
\end{gather*}
where 
\begin{align*}
\Xi _B=\{\varrho \geq 0 : \text{supp}\, \varrho \subset B\ \text{and}\ G \varrho \leq
1\}.
\end{align*}
For all of this we refer to \textsc{spitzer} \cite[Chapter VI]{Spitzer}.


\subsection*{Test of massiveness}

Assume that the Green function $G(x)$ is of the form: 
\begin{align}
G(x)= \frac{a(x)}{\chi (\Vert x\Vert)},\qquad x\neq 0,  \label{assump2}
\end{align}
where $\chi$ is a non-decreasing function satisfying the doubling condition 
\begin{align}
\chi (2 \theta)\leq C\chi (\theta),\quad \text{for all}\ \theta >0\ \text{%
and some}\ C>1 ,  \label{Delta2}
\end{align}
and $c_1\leq a(x)\leq c_2$ for some $c_1,c_2>0$ uniformly in $x$.

For a set $B$ define the following sequence of sets 
\begin{align*}
B_k=\{x\in B : 2^k\leq \Vert x\Vert <2^{k+1}\},\quad k=0,1,\ldots .
\end{align*}

\begin{thm}
\label{Wiener} A set\ $B$\ is non-massive if and only if 
\begin{gather*}
\sum _{k=0}^\infty \frac{Cap(B_k)}{\chi (2^k)} <\infty .  \label{a.w.t.1}
\end{gather*}
\end{thm}

To prove this statement, crucial in fact in our study, we use the
assumptions (\ref{assump2}) and (\ref{Delta2}) and follow step by step the
classical proof by Spitzer \cite[Section 26, T1]{Spitzer}.

\begin{exa}
Let $S$ be the simple random walk in $\gZ ^3$. The set $B=\mathbb{Z} _+ \times
\{0\}\times \{0\}$ is $S$-massive. Moreover, its proper subset $\mathcal{P}
\times \{0\}\times \{0\} $, where $\mathcal{P}$ is the set of primes, is
massive, see \cite{McKean}, \cite{McKean1}.

Let $0<\alpha <2$ and $S_\alpha$ be the $\alpha$-stable random walk in $\gZ ^3$. We claim that the set $B=\mathbb{Z} _+ \times \{0\}\times
\{0\}$ is not massive. To prove the claim we apply Theorem \ref{Wiener}
with $\chi (\theta)=\theta ^{3-\alpha}$. Let $|B_k|$ be the cardinality of\ $%
B_k$. Since $Cap(B_k)\leq |B_k|$, we have 
\begin{align*}
\sum _{k=0}^\infty \frac{Cap (B_k)}{\chi (2^{k+1})}&\leq \sum _{k=0}^\infty 
\frac{|B_k|}{2^{(k+1)(3-\alpha)}} \\
&\leq \sum _{k=0}^\infty \, \sum _{n:\, (n,0,0) \in B_k} \frac{1}{%
n^{3-\alpha}}= \sum _{n=1}^{\infty}\frac{1}{n^{3-\alpha}}<\infty .
\end{align*}
\end{exa}

\begin{exa}
Let $B$ be the hyperplane $\{ x\in \mathbb{Z} ^d:\, x_1=0\}$, $d\geq 3$. We
claim that

\begin{itemize}
\item[(i)] \ If\ $0< \alpha <1$, then\ $B$\ is a non-massive set with
respect to $S_\alpha$;

\item[(ii)] \ If\ $1\leq \alpha \leq 2$, then $B$ is a massive set with
respect to $S_\alpha$.
\end{itemize}
\end{exa}

Let $s_\alpha (n)$ be the projection of $S_\alpha (n)$ on the $x_1$-axis.
Evidently the set $B$ is $S_\alpha$-massive if and only if the random walk $%
\{s_\alpha (n)\}$ is reccurent. The characteristic function of the random
variable $S_{\alpha}(1)$ is 
\begin{align*}
H_\alpha (\theta)=1-(1-\frac{1}{d}\sum _{j=1}^{d}\cos \theta _j)^{\alpha
/2},\quad \theta \in \mathbb{R} ^d .
\end{align*}
It follows that the characteristic function $h_{\alpha}(\xi)$ of $s_\alpha
(1)$ is 
\begin{align*}
h_{\alpha}(\xi)=1-d^{-\alpha /2}(1-\cos \xi)^{\alpha /2},\quad \xi \in 
\mathbb{R} .
\end{align*}
Let $p(n)$ be the probability of return to $0$ in $n$ steps defined by the
random walk $\{ s_\alpha (n)\}$, then taking the inverse Fourier transform
we obtain 
\begin{gather*}
p(n)=\frac{1}{2\pi}\int _{-\pi}^\pi \big(h_{\alpha} (\xi)\big)^n\mathrm{d}
\xi .
\end{gather*}
It follows that 
\begin{align*}
\sum _{n\geq 0}p(n)&=\frac{1}{2\pi}\int _{-\pi}^\pi \frac{\mathrm{d} \xi}{%
1-h_{\alpha}(\xi)} \asymp \int _{0}^1 \frac{\mathrm{d} \xi}{\xi ^{\alpha }}
<\infty
\end{align*}
if and only if $0<\alpha <1$. By the well known criterion of transience, $%
s_\alpha (n)$ is transient. 

\section{Thorns}

In this section we assume that the dimension $d$ of the lattice $\mathbb{Z}
^d$ satisfies $d\geq 3$. For $x=(x_1,\ldots ,x_{d-1},x_d)$ we set $%
x^{\prime}=(x_1,\ldots ,x_{d-1})$ and write $x=(x^{\prime}, x_d)$. The 
\textit{thorn} $\mathcal{T}$ is defined as 
\begin{align*}
\mathcal{T}=\{ (x^{\prime}, x_d)\in \mathbb{Z} ^d: \Vert x^{\prime}\Vert
\leq t(x_d),\, x_d\geq 1 \},
\end{align*}
where $t(n)$ is a non-decreasing sequence of positive numbers. We study $S_\alpha$-massiveness of $\mathcal{T}$.

The problem of massiveness of thorns with respect to the simple random walk was studied in \textsc{It%
ô and McKean} \cite{McKean}. 
When $d=3$ the thorn $%
\mathcal{T}$ is $S$-massive, because the straight line is $S$-massive.
Whence for the simple random walk one assumes that $d\geq 4$.

By $Cap_\alpha (B)$ we denote the $S_\alpha$-capacity of the set $B\subset 
\mathbb{Z} ^d$, whereas $\widetilde{Cap_\alpha} (A)$ stands for the capacity
of the set $A\subset \mathbb{R} ^d$, associated with the rotationally invariant $\alpha$%
-stable process. 

\begin{prop}
\label{propMassive} Assume that 
$\limsup _{n\rightarrow \infty} t(n)/n = \delta >0,$
then the thorn $\mathcal{T}$ is $S_\alpha$-massive for any $0< \alpha < 2$
and $d\geq 3$.
\end{prop}

\begin{proof}
The sequence $t(n)$ is non-decreasing, whence by the assumption, $\limsup _{n\rightarrow \infty }\frac{t(2^n)}{2^n}\geq \delta /2$. Hence $t(2^n)/2^n >\delta /3$ for infinitely many $n$. For such 
$n$ consider the following sets 
\begin{align}
\mathcal{T}_n=\mathcal{T}\cap \{ x\in \mathbb{Z} ^d:\, 2^n\leq \Vert x\Vert
< 2^{n+1} \}.  \label{thornPart}
\end{align}
Let $B_n$ be the ball of radius $\delta 2^{n-2}$ centred at $(0,\ldots
,0,3\cdot 2^{n-1})$, see Figure 1. Since $t(2^n)>\delta /3\cdot 2^n>\delta 2^{n-2}$, we have $%
B_n\subset \mathcal{T}_n$, whence 
\begin{align*}
Cap_\alpha (\mathcal{T}_n)\geq Cap_\alpha (B_n).
\end{align*}
By the inequality (\ref{alphaCapBall}), Section 5, for some $c>0$, 
\begin{align*}
Cap_\alpha (B_n)\geq c2^{(n-2)(d-\alpha)}.
\end{align*}
It follows that 
\begin{align*}
\sum _{n\geq 0} \frac{Cap_\alpha (\mathcal{T}_n)}{2^{n(d-\alpha)}}=\infty .
\end{align*}
By Theorem \ref{Wiener}, the thorn $\mathcal{T}$ is massive.
\end{proof}
\begin{rem}\label{psi_thorn}
Using capacity bounds given in Section 5 and
following the same line of reasons as in the proof of Proposition \ref%
{propMassive}, we show that the thorn $\mathcal{T}$ satisfying $\limsup
t(n)/n>0$ is $S_\psi$-massive, for any special Bernstein function $\psi$ which satisfy the assumptions in Theorem \ref{Asymptotic}.
When $\lim t(n)/n=0$, $S_\psi$-massiveness of the thorn $\mathcal{T}$ is a
delicate question. In such a
generality this question is opened at present.
\end{rem}

\begin{figure}[h]
\begin{tikzpicture}[scale=0.8]
									\draw (0,0) circle [radius=0.02];
									\draw [densely dashed](0,0) ++(140:3) arc (140:40:3);
										\draw [densely dashed](0,0) ++(140:6) arc (140:40:6);
									\draw[thick] (1,0) to [out=90, in=230] (3,7);
									\draw[thick] (-1,0) to [out=90, in=310] (-3,7);
									\begin{scope}
										\clip (-1,0) to [out=90, in=310] (-3,7)-- (3,7) to [out=230, in=90] 																			(1,0)--(-1,0);
										\clip (0:3cm) -- (0:6cm)
										arc (0:180:6cm) -- (180:3cm)
										arc (180:0:3cm) -- cycle;
										\fill[color=black, opacity=0.3] (-5,0) rectangle (5,7); 
									\end{scope}
										
									\draw [fill=myGrey2, opacity=0.3, black, thick] (0,4.5) circle [radius=1.3];
										\draw (-0.1,2.7) node [right] {{\tiny $2^n$}};
										\draw (-0.1,6.3) node [right] {{\tiny $2^{n+1}$}};
										\draw [dashed](0,0) +(0:1)arc (0:180:1 and 0.3);		
										\draw (0,0)  +(180:1) arc (180:360:1 and 0.3);	
										\draw [dashed](0,7) +(0:3)arc (0:180:3 and 0.4);		
										\draw (0,7)  +(180:3) arc (180:360:3 and 0.4);			
										
											\draw (-0.1,4.5) node [right] {{\tiny $\frac{3}{2}2^{n}$}};
									\draw [thick](-0.05,4.5)--(0.05,4.5);
									\draw [dashed](-4,3)--(4,3);
										\draw (1.1,3.25) node [right] {{\tiny $t(2^n)$}};
										\draw [dashed](-4,6)--(4,6);
										\draw (2.4,6.3) node [right] {{\tiny $t(2^{n+1})$}};
												\draw[<-] [thick] (1.3, 5.5)--(2.8,8);
										\draw [fill=black, opacity=0.3] (2.8,8.4) ellipse  (1 and 0.4);
										\draw (2.8,8) node [above] {{\tiny Set\ $\mathcal{T}_n$}};
									\draw [dashed](0,0)--(0,8);
													\draw [densely dashed] (0,4.5) ellipse (1.3 and 0.3);
													\draw (0, 4.5) +(180:1.3) arc (180:270:1.3 and 0.3);
													\draw (0, 4.5) +(270:0.3) arc (270:360:1.3 and 0.3);
											\draw [black] (0,4.5) circle [radius=1.3];
											\draw[<-] [thick] (-0.5, 5)--(-2.8,8);
											\draw [fill=black, opacity=0.3] (-2.8,8.4) ellipse  (1 and 0.4);
											\draw [fill=myGrey2, opacity=0.3, black] (-2.8,8.4) ellipse  (1 and 0.4);
											\draw (-2.8,8) node [above] {{\tiny Ball\ $B_n$}};
							\end{tikzpicture}
\caption{Ball inscribed in the thorn.}
\end{figure}
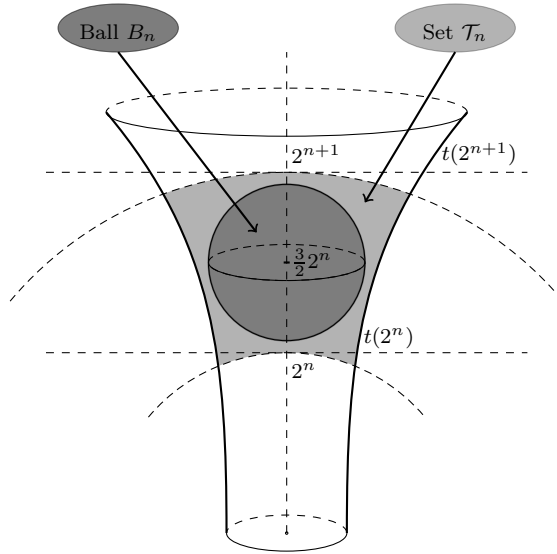
Next we study the case 
\begin{align}
\lim _{n\rightarrow \infty} \frac{t(n)}{n}=0.  \label{thinthorn}
\end{align}
Our reasons are based on the criterion of massiveness given in Theorem \ref%
{Wiener} but require more andvanced tools than those in the proof of
Proposition \ref{propMassive}. More precisely, we need upper and lower
bounds of the $\alpha$-capacity of non-spherically symmetric sets, long
cylinders for instance.

Let $\mathcal{F}_L$ be a cylinder of height $L$ with the unit disc as its
base, 
\begin{align*}
\mathcal{F}_L=\{ (x^{\prime},x_d)\in \mathbb{R} ^d\, :\, \Vert
x^{\prime}\Vert \leq 1,\, 0<x_d\leq L \}.
\end{align*}

\begin{prop}
\label{capCylinder} There exist constants $c_0, c_1>0$ which depend only on $%
d$ and $\alpha$ such that the following inequality holds 
\begin{align*}
c_0L\leq \widetilde{Cap_\alpha }(\mathcal{F}_L)\leq c_1L,\quad L\geq 1 .
\end{align*}
\end{prop}

\begin{proof}
Indeed, for the upper bound we write $L=k+m$, where $k=[L]$ and $m=L-[L]$.
Then, for some $c_1>0$, 
\begin{align*}
\widetilde{Cap_\alpha }(\mathcal{F}_L)\leq k\widetilde{Cap_\alpha} (\mathcal{%
F}_{1})+\widetilde{Cap_\alpha} (\mathcal{F}_m)\leq c_1 L.
\end{align*}
To obtain the lower bound we define the following sets 
\begin{align*}
D_i=\{ (x^{\prime},x_d)\in \mathbb{R} ^d\, :\, \Vert x^{\prime}\Vert \leq
1,\, i-1\leq x_d\leq i \},\quad i\geq 1.
\end{align*}
Let $\mu _i$ be the equilibrium measure of $D_i$, i.e. $\mu _i(D_i)=%
\widetilde{Cap _\alpha} (D_i)$. We have 
\begin{align*}
\widetilde{G_\alpha} \mu _{i+1}(x)=\widetilde{G_\alpha} \mu _1(x-ie_d),
\end{align*}
where $\widetilde{G_\alpha}$ is the Green function associated with the
symmetric $\alpha$-stable process in $\mathbb{R} ^d$ and $e_d=(0,0,\ldots
,1) $. Without loss of generality we can assume that $L$ is an integer
number. Define the following measure 
\begin{align*}
\sigma =\mu _1+\ldots + \mu _L.
\end{align*}
Clearly $\sigma (\mathbb{R} ^d)=L\widetilde{Cap_\alpha} (D_1)$. We claim
that 
\begin{align}
\widetilde{G_\alpha} \sigma\leq K<\infty .  \label{lemmaI}
\end{align}
Indeed, we have $G_\alpha \mu _1(x)\leq 1$, for all $x$, and 
\begin{align*}
\lim _{\Vert x\Vert \rightarrow \infty }\Vert x\Vert ^{d-\alpha}\widetilde{%
G_\alpha} \mu _1(x)<C,
\end{align*}
for some constant $C>0$. It follows that 
\begin{align*}
\sum _{i>0}\widetilde{G_\alpha} \mu _1(x-ie_d)\leq C\sum _{i>0}\Vert
x-ie_d\Vert ^{\alpha -d} \wedge 1.
\end{align*}
Observe that the series above converges uniformly in $x$ which proves the
claim. The inequality (\ref{lemmaI}) in turn implies the lower bound 
\begin{align*}
\widetilde{Cap_\alpha} (F_L)\geq \sigma (F_L)/K=\frac{L}{K}\widetilde{%
Cap_\alpha} (D_1).
\end{align*}
The proof is finished.
\end{proof}

Define the following sets 
\begin{align*}
\mathcal{F}_n^-&= \{ (x^{\prime},x_d)\in \mathbb{R} ^d:\, \Vert
x^{\prime}\Vert \leq t(2^n),\ \frac{4}{3}2^n\leq x_d < \frac{3}{4}2^{n+1} \};
\\
\mathcal{F}_n^+&= \{ (x^{\prime},x_d)\in \mathbb{R} ^d:\, \Vert
x^{\prime}\Vert \leq t(2^{n+1}),\ \frac{3}{4}2^n\leq x_d<\frac{4}{3}2^{n+1}
\}; \\
\begin{split}
\qquad F_n^-=\mathcal{F}_n^-\cap \mathbb{Z} ^d\quad \text{and}\quad F_n^+=%
\mathcal{F}_n^+\cap \mathbb{Z} ^d.
\end{split}%
\end{align*}
Let $Q(b)$ be the cube $[0,1]^d$ centered at $b$. For any set $B\subset 
\mathbb{Z} ^d$, we denote by $\widetilde{B}$ the subset of $\mathbb{R} ^d$
defined as 
\begin{align}
\widetilde{B}=\bigcup _{b\in B}Q(b).  \label{tildeSet}
\end{align}

\begin{thm}
\label{propThorns} Under the assumption (\ref{thinthorn}), the thorn $%
\mathcal{T}$ is $S_\alpha$-massive if and only if the series 
\begin{align}
\sum _{n>0}\Big(\frac{t(2^n)}{2^n}\Big)^{d-\alpha -1}  \label{szereg1}
\end{align}
diverges.
\end{thm}

Before embarking on the proof of Theorem \ref{propThorns} we illustrate the
statement by the following example. Consider the thorn $\mathcal{T}$ with $%
t(n)=n/\left(\log (1+n)\right)^\beta$, $\beta >0$. Then $\mathcal{T}$ is $%
S_\alpha$-massive if and only if $\beta \leq 1/(d-\alpha -1)$.

\begin{proof}
Assume that the series (\ref{szereg1}) is convergent. Show that the set $%
\mathcal{T}$ is non-massive. For any compact set $A\subset \mathbb{R} ^d$
and for any $s>0$ the following scaling property holds 
\begin{align}
\widetilde{Cap_{\alpha}}(sA)= s^{d-\alpha}\widetilde{Cap_{\alpha}} (A),
\label{scaleInvariance}
\end{align}
see e.g. \textsc{Sato} \cite[Example 42.17]{Sato}. Using Proposition \ref%
{capCylinder}, the assumption (\ref{thinthorn}) and the equation (\ref%
{scaleInvariance}), for enough large $n$ we have 
\begin{align*}
\widetilde{Cap_{\alpha}} (\mathcal{F}_n^+)&=\widetilde{Cap_{\alpha}} \big( %
t(2^{n+1})\!\cdot \! \mathcal{F}_n^+/ t(2^{n+1}) \big)  \notag \\
&= t(2^{n+1})^{d-\alpha}\!\cdot \!\widetilde{Cap_{\alpha}} \big( \mathcal{F}%
_n^+/t(2^{n+1}) \big)  \notag \\
&\leq c_1 t(2^{n+1})^{d-\alpha }\!\cdot \!t(2^{n+1})^{-1}\!\cdot \! \Big(%
\frac{4}{3}2^{n+1}- \frac{3}{4}2^n\Big)  \notag \\
&\leq c_2 t(2^{n+1})^{d-\alpha -1}\!\cdot \!2^{n+1},
\end{align*}
for some $c_1, c_2>0$. Let $\mathcal{T}_n$ be as in (\ref{thornPart}). Since 
$\mathcal{T}_n \subset F_n^+$, see Figure 2, 
\begin{align*}
Cap_\alpha (\mathcal{T}_n)\leq Cap_\alpha (F_n^+).
\end{align*}
By Theorem \ref{CapContin}, Section 5, 
\begin{align}
c_3 \widetilde{Cap_\alpha}(\widetilde{F}_n^+)\leq Cap_\alpha (F_n^+)\leq c_4 
\widetilde{Cap_\alpha}(\widetilde{F}_n^+),  \label{compCyl1}
\end{align}
for some $c_3, c_4>0$. Using again Proposition \ref{capCylinder} we obtain 
\begin{align}
c_5 \widetilde{Cap_\alpha} (\mathcal{F}_n^+)\leq \widetilde{Cap_\alpha} (%
\widetilde{F}_n^+)\leq c_6\widetilde{Cap_\alpha} (\mathcal{F}_n^+),
\label{compCyl2}
\end{align}
for some $c_5,c_6>0$. All the above show that 
\begin{align*}
\sum _{n>0} \frac{Cap_\alpha (\mathcal{T}_n)}{2^{n(d-\alpha)}}&\leq c_7\sum
_{n>0}\Big(\frac{t(2^{n+1})}{2^{n+1}}\Big)^{d-\alpha -1}<\infty ,
\end{align*}
as desired.

Conversely, assume that the series (\ref{szereg1}) is divergent. Show that
the set $\mathcal{T}$ is massive. Applying Proposition \ref{capCylinder},
the assumption (\ref{thinthorn}) and the equation (\ref{scaleInvariance}) we
have 
\begin{align*}
\widetilde{Cap_{\alpha}} (\mathcal{F}_n^-)&=\widetilde{Cap_{\alpha}} \big( %
t(2^{n})\!\cdot \! \mathcal{F}_n^-/ t(2^{n}) \big)  \notag \\
&= t(2^{n})^{d-\alpha}\!\cdot \!\widetilde{Cap_{\alpha}} \big( \mathcal{F}%
_n^-/t(2^{n}) \big)  \notag \\
&\geq c^{\prime}_1 t(2^{n})^{d-\alpha }\!\cdot \!t(2^{n})^{-1}\!\cdot \! %
\Big(\frac{3}{4}2^{n+1}- \frac{4}{3}2^n\Big)  \notag \\
&\geq c^{\prime}_2 t(2^{n})^{d-\alpha -1}\!\cdot \!2^{n},
\end{align*}
for some $c^{\prime}_1, c^{\prime}_2>0$. Since $F_n^-\subset \mathcal{T}_n$,
see Figure 2, 
\begin{align*}
Cap_\alpha (\mathcal{T}_n)\geq Cap_\alpha (F_n^-).
\end{align*}
Similarly to (\ref{compCyl1}) and (\ref{compCyl2}) we get 
\begin{align*}
c^{\prime}_3 \widetilde{Cap_\alpha}(\widetilde{F}_n^-)\leq Cap_\alpha
(F_n^-)\leq c^{\prime}_4 \widetilde{Cap_\alpha}(\widetilde{F}_n^-)
\end{align*}
and 
\begin{align*}
c^{\prime}_5 \widetilde{Cap_\alpha} (\mathcal{F}_n^-)\leq \widetilde{%
Cap_\alpha} (\widetilde{F}_n^-)\leq c^{\prime}_6\widetilde{Cap_\alpha} (%
\mathcal{F}_n^-),
\end{align*}
for some constants $c^{\prime}_3,c^{\prime}_4,c^{\prime}_5,c^{\prime}_6>0$.
Thus, at last, 
\begin{align*}
\sum _{n>0} \frac{Cap_\alpha (\mathcal{T}_n)}{2^{n(d-\alpha)}}&\geq
c^{\prime}_7\sum _{n>0}\Big(\frac{t(2^{n})}{2^{n}}\Big)^{d-\alpha -1}=\infty
,
\end{align*}
as desired.
\end{proof}

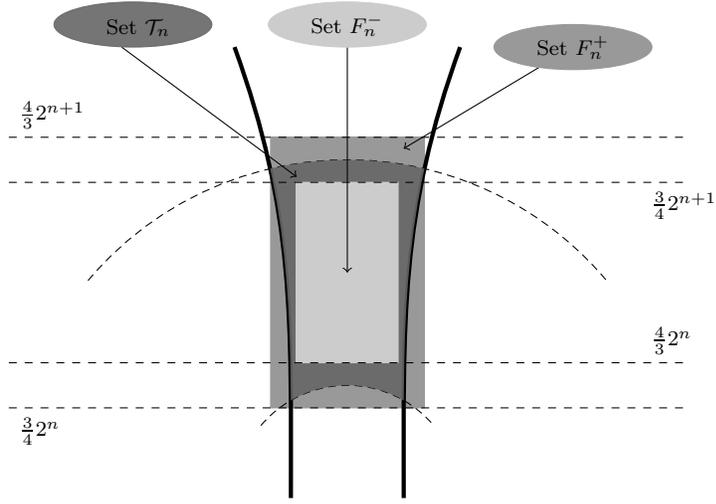
\begin{figure}[h]
\begin{tikzpicture}[scale=1.5]
									\draw [fill=myGrey3, myGrey3](-0.68,0.8) rectangle (0.68,3.2);		
									\draw [densely dashed](0,0) ++(140:1) arc (140:40:1);	
										\draw [densely dashed](0,0) ++(140:3) arc (140:40:3);	
									\draw[ultra thick] (0.5,0) to [out=90, in=250] (1,4);				
									\draw[ultra thick] (-0.5,0) to [out=90, in=290] (-1,4);			
									\begin{scope}													
										\clip (-0.5,0) to [out=90, in=290] (-1,4)-- (1,4) to [out=250, in=90] 																			(0.5,0)--(-0.5,0);
										\clip (0:1cm) -- (0:3cm)
										arc (0:180:3cm) -- (180:1cm)
										arc (180:0:1cm) -- cycle;
										\fill[color=myGrey2, opacity=0.9] (-3,0) rectangle (3,5); 
									\end{scope}	
									\draw [fill=myGrey4, myGrey4](-0.45,1.2) rectangle (0.45,2.8);			
									\draw (-3,0.6) node [right] {{\tiny $\frac{3}{4}2^n$}};
										\draw (-3,3.4) node [right] {{\tiny $\frac{4}{3}2^{n+1}$}};
										\draw [dashed](-3,0.8)--(3,0.8);
									\draw [dashed] (-3,3.2)--(3,3.2);
									\draw [dashed] (-3,1.2)--(3,1.2);
									\draw [dashed] (-3,2.8)--(3,2.8);
									\draw (2.6,2.6) node [right] {{\tiny $\frac{3}{4}2^{n+1}$}};
										\draw (2.6,1.4) node [right] {{\tiny $\frac{4}{3}2^{n}$}};
										\draw[<-] (0.5,3.1)--(2,4);
										\draw [fill=myGrey3, myGrey3] (2,4) ellipse (0.7 and 0.2);
										\draw (2,3.8) node [above] {{\tiny Set $F_n^+$}}; 
									\draw[<-]  (-0.45, 2.85)--(-2,4);
								    \draw [fill=myGrey2, opacity=0.9, myGrey2] (-1.9,4.2) ellipse  (0.7 and 0.2); 														\draw (-1.85,4) node [above] {{\tiny Set\ $\mathcal{T}_n$}};																		
									\draw[<-]  (0, 2)--(0,4);
								    \draw [fill=myGrey4, myGrey4] (0,4.2) ellipse  (0.7 and 0.2); 																		\draw (0,4) node [above] {{\tiny Set\ $F_n^-$}};		
							\end{tikzpicture}
\caption{Two cylinders inscribed in and circumscribed around the thorn.}
\end{figure}

\newpage
\section{Two comparisons.}

Let $\psi $ be a special Bernstein function (see Remark \ref{psi_thorn}). Let $B_\psi $ be a Lévy process in $\mathbb{R} ^d$ obtained by subordination
of the Brownian motion $B$. Let $S_\psi$ be the random walk obtained by
subordination of the simple random walk $S$. Let $\widetilde{G_\psi} (x)$
(resp. $G_\psi$) be the Green function of $B_\psi$ (resp. $S_\psi$). 
\par In what follows we assume
that $\psi \in \mathcal{BF}$ satisfies the conditions of Theorem \ref{Asymptotic}. 

\begin{prop}
\label{compareGreen} The function $\widetilde{G_\psi} (x)$ has the following
asymptotic 
\begin{gather*}
\widetilde{G_\psi} (x)\sim \frac{A_{d,\alpha}}{\Vert x\Vert
^{d}\psi (1/\Vert x\Vert ^2)} ,\quad x\to \infty ,
\end{gather*}
where 
\begin{align*}
A_{d,\alpha}=\frac{\Gamma ((d-\alpha )/2)}{2^\alpha \!\cdot \!
\pi ^{d/2}\!\cdot \! \Gamma (\alpha /2)}.
\end{align*}
In particular, 
\begin{align*}
\widetilde{G_\psi }(x)\sim (2/d)^{\alpha /2} G_\psi (x),\quad  x\to \infty .
\end{align*}
\end{prop}

\begin{proof}
As $\psi$ is a special Bernstein function, the potential
measure $U$ associated with the corresponding (continuous time) subordinator has a monotone
density $u(t)$, see e.g. \cite[Chapter V, Theorem 5.1]{Bogdan}.
Since $\mathcal{L}(U)(\lambda)=1/\psi (\lambda)$, the Karamata theorem
implies that the density function $u(t)$ satisfies 
\begin{gather*}
u(t)\sim \frac{ 1}{\Gamma (\alpha /2)}t^{\alpha /2-1}l(t)\quad \text{at}\
\infty .
\end{gather*}
Recall that, by definition, 
\begin{align*}
\widetilde{G_\psi} (x)&=\int _0^\infty (4\pi t)^{-d/2}\exp \left\{{-\frac{%
\Vert x\Vert ^2}{4t}}\right\}\, u(t)\, \mathrm{d} t,
\end{align*}
whence, as $\Vert x \Vert \rightarrow \infty$ 
\begin{align*}
\widetilde{G_\psi} (x)&=4^{-1}\pi ^{-d/2}\Vert x\Vert ^{2-d} \int _0^\infty
s^{d/2-2}e^{-s}u\Big( \frac{\Vert x\Vert ^2}{4s} \Big)\mathrm{d} s \\
&\sim 4^{-1}\pi ^{-d/2}\Vert x\Vert ^{2-d} \int _0^\infty
s^{d/2-2}e^{-s}u(\Vert x\Vert ^2)\Big( \frac{1}{4s} \Big)^{\alpha /2-1}%
\mathrm{d} s \\
&=2^{-\alpha}\pi ^{-d/2}\Vert x\Vert ^{2-d} u(\Vert x\Vert ^2) \int
_0^\infty s^{d/2-\alpha /2-1}e^{-s}\mathrm{d} s \\
&\sim \frac{\Gamma (\frac{d-\alpha}{2})}{2^{\alpha}\! \cdot \! \cdot \! \pi
^{d/2}\! \cdot \! \Gamma (\alpha /2)}\Vert x\Vert ^{\alpha -d}l(\Vert x\Vert
^2) .
\end{align*}
Combining this result with that of Theorem \ref{Asymptotic} we obtain the
claimed comparison of Green functions $\widetilde{G_\psi}$ and $G_\psi$.
\end{proof}


Let $\widetilde{Cap_\psi }(A)$ be the capacity of a set $A\subset \mathbb{R}
^d$ associated with the process $B_\psi$. Recall that by definition (see
e.g. \cite{Blumental}) 
\begin{align*}
\widetilde{Cap_\psi }(A)=\sup \, \{\mu (A):\, \mu \in \mathcal{K}_A\},
\end{align*}
where $\mathcal{K}_A$ is the class of measures supported by $A$ and such
that 
\begin{align*}
\widetilde{G_\psi} \mu (\xi)=\int _A \widetilde{G_\psi} ( \xi - \eta ) \mu (%
\mathrm{d} \eta)\leq 1,\quad \text{for all}\ \xi \in \mathbb{R} ^d .
\end{align*}
Let $Cap_\psi (B)$ be the capacity of a set $B\subset \mathbb{Z} ^d$
associated with the process $S_\psi$. Similarly 
\begin{gather*}
Cap_\psi (B)=\sup\, \{\, \sum _{y\in B}\phi (y):\, \phi \in \Xi _B \} ,
\end{gather*}
where 
\begin{align*}
\Xi _B=\{\phi \geq 0 : \text{supp}\, \phi \subset B\ \text{and}\ G_\psi \phi
\leq 1\}.
\end{align*}

\begin{thm}
\label{CapContin} Let $B$ be a bounded subset of $\mathbb{Z} ^d$. Let $\widetilde{B}$ be defined at (\ref{tildeSet}). There exist
constants $c_1,\, c_2>0$, which depend only on $d$ and $\psi$, and such that 
\begin{gather*}
c_1\widetilde{Cap_\psi} (\widetilde{B})\leq Cap_\psi (B)\leq c_2\widetilde{%
Cap_\psi} (\widetilde{B}).
\end{gather*}
\end{thm}

\begin{proof}
Take $a,b \in B$. Let $Q(a)$ be the cube $[0,1]^d$
centered at $a\in B$. Let $\mathrm{d} \eta$ be the Lebesgue measure in $%
\mathbb{R} ^d$. By Proposition \ref{compareGreen} and radial monotonicity of $\widetilde{G_\psi}$, we can find a constant $%
c_2>0$ which does not depend on $a$ and $b$, and such that for $\xi \in Q(a)$
and $\eta \in Q(b)$, 
\begin{gather}  \label{eq8}
\int _{Q(b)}\widetilde{G_\psi} ( \xi -\eta ) \mathrm{d} \eta \leq c_2 G_\psi
(a-b).
\end{gather}
Let $E$ be the equilibrium distribution of $B$ associated with the random
walk $S_\psi$. We define a new measure 
\begin{gather*}
\mathrm{d} \nu ( \eta)=\sum _{b\in B}E(b)\mathbf{1}_{Q(b)}(\eta)\mathrm{d}
\eta .
\end{gather*}
Using (\ref{eq8}) we compute the potential $\widetilde{G_\psi}\nu$ 
\begin{align*}
\int _{\widetilde{B}} \widetilde{G_\psi} ( \xi -\eta )\, \mathrm{d} \nu (
\eta)&=\sum _{b\in B}\int _{Q(b)}\widetilde{G_\psi} (\xi -\eta ) E(b)\, 
\mathrm{d} \eta \\
&\leq c_2\sum _{b\in B}G_\psi(a-b)E(b)=c_2G_\psi E(a)\leq c_2.
\end{align*}
Thus, the measure $c_2^{-1}\nu $ belongs to the class $\mathcal{K}_{%
\widetilde{B}}$, therefore 
\begin{gather}  \label{eq9}
\widetilde{Cap_\psi} (\widetilde{B})\geq \frac{1}{c_2}\nu (\widetilde{B}).
\end{gather}
On the other hand 
\begin{gather}  \label{eq10}
\nu (\widetilde{B})=\int _{\widetilde{B}}\mathrm{d} \nu ( \eta)=\sum _{b\in
B}\int _{Q(b)}E(b)\, \mathrm{d} \eta=\sum _{b\in B}E(b)=Cap_\psi (B).
\end{gather}
Combining (\ref{eq9}) and (\ref{eq10}) we obtain 
\begin{gather*}
Cap_\psi (B)=\nu (\widetilde{B})\leq c_2\, \widetilde{Cap_\psi} (\widetilde{B%
}).
\end{gather*}
\par For the converse we use again Proposition \ref%
{compareGreen} and radial monotonicity of $\widetilde{G_\psi}$. Let $a,b \in B$. Choose $c_1>0$, which does not depend on $a$ and $b$,
such that for $\xi \in Q(a)$, $\eta \in Q(b)$, 
\begin{gather}  \label{eq11}
c_1\, G_\psi (a-b)\leq \widetilde{G_\psi} (\xi -\eta ).
\end{gather}
Let $\widetilde{E}$ be the equilibrium measure of $\widetilde{B}$, i.e. $%
\widetilde{Cap_\psi} (\widetilde{B})=\widetilde{E}(\widetilde{B})$. Define a
distribution $\varrho$ supported by the set $B$ as 
\begin{gather*}
\varrho (b)=\widetilde{E}(Q(b)),\quad b\in B.
\end{gather*}
Let 
\begin{gather*}
p=c_1G_\psi \varrho .
\end{gather*}
Using (\ref{eq11}) we get 
\begin{gather*}
p(a) \leq \sum _{b\in B}\widetilde{G_\psi }( \xi -\eta ) \varrho (b)\leq
\int _{\widetilde{B}} \widetilde{G_\psi} ( \xi -\eta ) \, \mathrm{d} 
\widetilde{E}( \eta)\leq 1.
\end{gather*}
It follows that $c_1 \varrho \in \Xi _B$, whence 
\begin{gather}  \label{eq12}
Cap_\psi (B)\geq c_1\, \varrho (B).
\end{gather}
Computing $\varrho (B)$ we obtain 
\begin{gather}  \label{eq13}
\varrho (B)=\sum _{b\in B}\varrho (b)=\sum _{b\in B}\int _{Q(b)}\mathrm{d} 
\widetilde{E} (\eta )=\int _{\widetilde{B}}\mathrm{d} \widetilde{E}( \eta )=%
\widetilde{E}(\widetilde{B}).
\end{gather}
From (\ref{eq12}) and (\ref{eq13}) we deduce that 
\begin{gather*}
Cap_\psi (B)\geq c_1\, \varrho (B)=c_1\,\widetilde{E}(\widetilde{B} )=c_1\, 
\widetilde{Cap_\psi} (\widetilde{B}).
\end{gather*}
The proof is finished.
\end{proof}

\begin{cor}
\label{CorBalls} Let $B(0,r)\subset \mathbb{Z} ^d$ be a ball of radius $r>0$
centered at $0$. The following inequalities hold 
\begin{align*}
cr^{d}\psi (1/r^2)\leq Cap_\psi (B(0,r)) \leq Cr^{d}\psi (1/r^2),
\end{align*}
for some constants $c,C>0$ and all $r>0$. In particular, 
\begin{align}
cr^{d-\alpha}\leq Cap_\alpha (B(0,r))\leq Cr^{d-\alpha}.
\label{alphaCapBall}
\end{align}
\end{cor}
\begin{proof}
Assume $d\geq 3$. Let $\phi $ be the L\'{e}vy exponent of $B_\psi$, that is
\begin{align*}
\gE e^{i\xi B_\psi (t)}=e^{-t\phi (\xi)},\quad \xi \in \gR ^d.
\end{align*} 
Since $B_\psi$ is a subordinated Brownian motion, we have
\begin{align*}
\phi (\xi)=\psi (\Vert \xi \Vert ^2),\quad \xi \in \gR ^d.
\end{align*}
The function $\phi (s)$ is increasing whence
\cite[Proposition 3]{Grzywny} applies in the form
\begin{align*}
\widetilde{Cap_\psi}  (B(0,r))\asymp \psi (r^{-2})r^d.
\end{align*}
At last, Theorem \ref{CapContin} yields the desired result.
\par When $d\leq 2$ we proceed as follows. We use \cite[Proposition 5.55]{Bogdan},
\begin{align*}
\widetilde{Cap_\psi}  (B(0,r))\asymp\frac{r^d}{\int _{B(0,r)}\widetilde{G_\psi}(x)\ud x},
\end{align*}
and \cite[Proposition 5.56]{Bogdan},
\begin{align*}
\int _{B(0,r)}\widetilde{G_\psi}(x)\asymp \gE^0\tau _{B(0,r)},
\end{align*}
where $\tau _{B(0,r)} $ is  $B_\psi$--first exit time from the ball $B(0,r)$.
We use \cite[Theorem 1 and p. 954]{Pruitt},
\begin{align*}
\gE^0\tau _{B(0,r)}\asymp \frac{1}{h(r)},
\end{align*}
where 
\begin{align*}
h(r)=\int _{\gR ^d}\Big(\frac{\Vert x\Vert ^2}{r^2}\wedge 1\Big)\ud \nu (x)
\end{align*}
and $\nu$ is the L\'{e}vy measure associated with the L\'{e}vy exponent $\phi$, see \cite[Section 3]{Pruitt}.
By \cite[Corollary 1]{Grzywny},
\begin{align*}
h(r)\asymp \psi (r^{-2}).
\end{align*}
The proof is finished.
\end{proof}

\subsection*{Acknowledgements}

This paper was started at Wrocław University and finished at Bielefeld
University (SFB-701). We thank A.~Grigor'yan, W.~Hansen,
S.~Molchanov and Z.~Vondraček for fruitful discussions. We also
thank the anonymous referee for valuable remarks.%

\end{document}